\documentclass[11pt]{amsart}
\usepackage{mathrsfs,amsthm,amssymb,amsmath,url,a4wide}
\usepackage[latin1]{inputenc}
\usepackage{csquotes}

\theoremstyle{plain}
    \newtheorem{thm}{Theorem}

    \newtheorem{prop}[thm]{Proposition}
    \newtheorem{cor}[thm]{Corollary}

    \newtheorem{prob}[thm]{Problem}
    \newtheorem{conj}[thm]{Conjecture}
    \newtheorem{ex}[thm]{Example}
    
\theoremstyle{definition}

\theoremstyle{remark}

\newcommand{\ignore}[1]{}

\title{A note on X-rays of permutations and a problem of Brualdi and Fritscher}
\author{Gustav Nordh} 

\email{gustav.nordh@gmail.com}
\begin{document}
\begin{abstract}
The subject of this note is a challenging conjecture about X-rays of permutations which is a special case of a conjecture regarding Skolem sequences. In relation to this, Brualdi and Fritscher [Linear Algebra and its Applications, 2014] posed the following problem: Determine a bijection between extremal Skolem sets and binary Hankel X-rays of permutation matrices. We give such a bijection, along with some related observations.
\end{abstract}
\maketitle
\section{Skolem sequences}
Skolem sequences originates from the work by Thoralf Skolem in 1957~\cite{Skol} on the construction of Steiner triple systems. 
Skolem proved that the set $\{1,2,\dots, 2n\}$ can be partitioned in $n$ pairs $(s_i,t_i)$ such that $t_i - s_i = i$ for $i=1,2,\dots,n$, if and only if $n \equiv 0,1 \pmod{4}$. This result can be reformulated as: There is a sequence  
with two copies of every element $k$ in 
$A=\{1,2,\dots,n\}$ such that the two copies of $k$ are placed $k$ places apart in the sequence, if and only if $n \equiv 0,1 \pmod{4}$. For example, the set $\{1,2,3,4\}$ can be used to form the sequence $42324311$, but the set $\{1,2,3\}$ cannot be used to form such a sequence. For more information on Skolem sequences and generalizations thereof, see the survey~\cite{HCD}.

A natural generalization is when the set of differences $A$ is any set or multiset of positive integers. If $A$ is a multiset, the sequences are called multi Skolem 
sequences, and the corresponding existence question
is that of deciding for which multisets 
$A=\{a_1,\dots,a_n\}$ there is a partition of $\{1,\dots,2n\}$ into
the differences in $A$. A set $A$ such that there is a partition of $\{1,\dots,2n\}$ into
the differences in $A$ is called a multi Skolem set. 
In my MSc thesis~\cite{EXJ}, I identified (rather obvious) parity and density conditions that are necessary for $A$ to be a multi Skolem set. These conditions were far from sufficient (unsurprisingly as the existence question for multi Skolem sets turns out to be NP-complete~\cite{Nordh09}). 

But, surprisingly, I discovered that when $A$ is an ordinary set (i.e., not a multiset), then these simple necessary conditions seem to be sufficient.
\begin{conj}[\cite{EXJ}]
A set $A = \{a_1,a_2,\dots,a_n\}$ with $a_1 < a_2 < \dots < a_n$ is a Skolem set if and only if the number of even $a_i$'s is even, and $\sum^n_{i=m} a_i \leq n^2-(m-1)^2$ for each 
$1 \leq m \leq n$. \label{perfconj}
\end{conj}
A particularly interesting special case of Conjecture~\ref{perfconj} emerge when the set of differences $A = \{a_1,a_2,\dots,a_n\}$ is as sparse as possible (in the sense that adding $1$ to any element in $A$ force $A$ to violate the density condition), i.e., $\sum^n_{i=1} a_i = n^2$. Such sets (multisets) $A = \{a_1,a_2,\dots,a_n\}$ satisfying $\sum^n_{i=1} a_i = n^2$ are called extremal.
\begin{conj}[\cite{Nordh08}]
A set $A = \{a_1,a_2,\dots,a_n\}$ with $a_1 < a_2 < \dots < a_n$ and $\sum^n_{i=1} a_i = n^2$ is an extremal Skolem set if and only if $\sum^n_{i=m} a_i \leq n^2-(m-1)^2$ for each $1 \leq m \leq n$.
\label{extconj}
\end{conj}
Note that the parity condition in Conjecture~\ref{perfconj} is implied by $\sum^n_{i=1} a_i = n^2$, and
that the conjecture is invalid for extremal multisets. A minimal counterexample is $A = \{4,4,4,8,8,8\}$. 

\section{X-rays of permutations}
The properties of X-rays of permutations of interest in this note are primarily investigated by Bebeacua et al.~\cite{Postnikov}, and Brualdi and Fritscher~\cite{BF14}, from where we borrow most of the notation.
Let $S_n$ denote the symmetric group of all permutations of $\{1,2,\dots,n\}$ and let $[P_{\pi}]$ denote the permutation matrix of $\pi \in S_n$. 

The diagonal X-ray of $\pi$ (called the Toeplitz X-ray in~\cite{BF14}) is denoted $d(\pi) = (d_1,d_2,\dots,d_{2n-1})$ and defined by $$d_k = \sum_{i-j = n-k} [P_{\pi}]_{i,j} \;\;\;\; (k=1,2,\dots,2n-1).$$

The antidiagonal X-ray of $\pi$ (called Hankel X-ray in~\cite{BF14}) is denoted $d'(\pi) = (d'_1,d'_2,\dots,d'_{2n-1})$ and defined by $$d'_k = \sum_{i+j = k+1} [P_{\pi}]_{i,j} \;\;\;\; (k=1,2,\dots,2n-1).$$

Hence, the diagonal (antidiagonal) X-ray of $\pi$ is the number of non-empty cells in each of the $2n-1$ upper left to lower right (upper right to lower left) diagonals of the permutation matrix $[P_{\pi}]$.
The Toeplitz characteristic of $\pi$ is denoted $$l_t(\pi) = (l_1,\dots,l_n) \;\;\;\; (1 \leq l_1 \leq \dots \leq l_n \leq 2n-1)$$ and defined to be the indices of the non-zero entries in $d(\pi)=(d_1,d_2,\dots,d_{2n-1})$ arranged in non-decreasing order, where an index $i$ is repeated $k$ times if its value $d_i$ is $k$.
The Hankel characteristic of $\pi$ is denoted $$l_h(\pi) = (l'_1,\dots,l'_n) \;\;\;\; (1 \leq l'_1 \leq \dots \leq l'_n \leq 2n-1)$$ and defined to be the indices of the non-zero entries in $d'(\pi)=(d'_1,d'_2,\dots,d'_{2n-1})$ arranged in non-decreasing order, where an index $i$ is repeated $k$ times if its value $d'_i$ is $k$.
X-rays that are $(0,1)$-vectors are called binary X-rays.

A tournament is a loopless digraph
such that for every two distinct vertices $i$ and $j$ either $(i,j)$ or $(j,i)$ is an arc. The score
vector of a tournament on $n$ vertices is the vector of length $n$ whose entries are the
out-degrees of the vertices of the tournament arranged in non-decreasing order.
Bebeacua et al.~\cite{Postnikov} state the following conjecture\footnote{Despite considerable meditation on the subject, I find no clear connection between binary diagonal X-rays of permutations and score vectors of tournaments, except that they seem to be equinumerous. Perhaps the following equivalent (see~\cite{BF14}) reformulation of Conjecture~\ref{conj:post} is more informative: $(a_1, a_2, \dots, a_{n})$ with $0 < a_1 < a_2 < \dots < a_n < 2n$ 
is the Toeplitz characteristic $l_t(\pi)$ of a $\pi \in S_n$, if and only if $\sum^k_{i=1} a_i \geq k^2$ for each $1 \leq k \leq n$, with equality holding for $k=n$.}.
\begin{conj}[\cite{Postnikov}] \label{conj:post}
The number of binary diagonal X-rays of $n \times n$ permutation matrices equals the number of score vectors of tournaments of order $n$.
\end{conj}

\section{A problem of Brualdi and Fritscher}
Brualdi and Fritscher~\cite{BF14} posed the following problem:
\enquote{Determine a bijection between perfect extremal Skolem sets of size $n$ and binary Hankel X-rays of $n \times n$ permutation matrices.}
In a footnote they state:
\enquote{In~\cite{Postnikov,Nordh08} it is stated that the two conjectures are equivalent but we fail to see any relation between the two except that both involve $n$ positive integers satisfying the same conditions.}
The two conjectures referred to in the above quote are Conjectures~\ref{extconj} and~\ref{conj:post} in this note.
In fact, a bijection between extremal Skolem sets and binary X-rays, that Brualdi and Fritscher asks for, is implicit in~\cite{Nordh08}. We proceed to give this bijection explicitly, along with some further observations. First notice that what is called a perfect extremal Skolem set in~\cite{BF14,Nordh08} is what is called a Skolem set in this note. Moreover, the bijection is stated for diagonal (Toeplitz) X-rays instead of antidiagonal (Hankel) X-rays. As there is an obvious bijection between diagonal and antidiagonal X-rays of permutations (by reversing all rows of the permutation matrix), this is equivalent.

\begin{prop}\label{prop:bij}
There is a bijection between extremal Skolem sets of size $n$ and binary diagonal (Toeplitz) X-rays of $n \times n$ permutation matrices. More explicitly, $A = \{a_1, a_2, \dots, a_{n}\}$ is an extremal Skolem set if and only if there is a $\pi \in S_n$ such that the non-zero components of $d(\pi) = (d_1,d_2,\dots,d_{2n-1})$ are $\{d_{a_1},d_{a_2},\dots,d_{a_n}\}$.
\end{prop}
\begin{proof}
$A = \{a_1, a_2, \dots, a_{n}\}$ is an
extremal Skolem set if and only if the set $\{1,2,\dots, 2n\}$ can be partitioned into the differences in $A$. Since
$A$ is extremal, this partition $a_i = t_i - s_i$, ($i = 1,2,\dots,n$) has the property that 
$s_i \in \{1,2,\dots,n\}$ and $t_i \in \{n+1,n+2,\dots,2n\}$, ($i = 1,2,\dots,n$). Hence, the permutation $\pi$ on 
$\{1,2,\dots,n\}$ defined as 
$\pi(s_i) = t_i - n$, ($i = 1,2,\dots,n$) has binary diagonal X-ray $d(\pi) = (d_1,d_2,\dots,d_{2n-1})$, 
with $d_i = 1$ if and only if $i \in \{t_1 - s_1, t_2 - s_2, \dots, t_n - s_n\} = \{a_1, a_2, \dots, a_{n}\}$.
\end{proof}
\begin{ex} Consider the permutation $\pi \in S_4$ given by $\pi(1) = 3, \pi(2) = 2, \pi(3) = 4, \pi(4) = 1$ represented by the $4 x 4$ upper right submatrix of the matrix below. The diagonal X-ray $d(\pi) = (d_1,d_2,d_3,d_4,d_5,d_6,d_7) = (1,0,0,1,1,1,0)$ with $d_i = 1$ if and only if $i \in \{1,4,5,6\}$. The permutation $\pi$ defines a partition of $\{1,2,3,4,5,6,7,8\}$ into pairs $(t_i,s_i)$ by $\pi(s_i) = t_i - n$, with corresponding set of differences $A = \{t_i - s_i\}$ ($i = 1,2,3,4$) (in this example $s_1=4, s_2 = 2, s_3 = 3, s_4 = 1$, and hence, the differences are $A=\{\pi(4) + 4 - 4,\pi(2) + 4 - 2, \pi(3) + 4 - 3, \pi(1) + 4 - 1\} = \{1,4,5,6\}$ with corresponding Skolem sequence (6,4,5,1,1,4,6,5)).
$$\left[ \begin{array}{cccc|cccc} s_4 & 0 & 0 & 0 & 0 & 0 & t_4 & 0\\ 0 & s_2 & 0 & 0 & 0 & t_2 & 0 & 0\\ 0 & 0 & s_3 & 0 & 0 & 0 & 0 & t_3\\ 0 & 0 & 0 & s_1 & t_1 & 0 & 0 & 0\\ \hline 6 & 4 & 5 & 1 & 1 & 4 & 6 & 5\end{array} \right]$$
\end{ex}

Note that the statement in Proposition~\ref{prop:bij} that the non-zero components of the diagonal X-ray $d(\pi) = (d_1,d_2,\dots,d_{2n-1})$ are $\{d_{a_1},d_{a_2},\dots,d_{a_n}\}$ is equivalent to the statement $l_t(\pi) = (a_1,\dots,a_n)$ where $l_t(\pi)$ is the Toeplitz characteristic of $\pi$. Thus, the result in Proposition~\ref{prop:bij} can be stated as:
\begin{cor}
$A = \{a_1, a_2, \dots, a_{n}\}$ is an extremal Skolem set if and only if there is a $\pi \in S_n$ with Toeplitz characteristic $l_t(\pi) = (a_1, a_2, \dots, a_{n})$.
\end{cor}
Also note that the proof of Proposition~\ref{prop:bij} hold with minor modifications for extremal multi Skolem sets and non-binary Toeplitz X-rays (for the details, see~\cite{Nordh08}).
\begin{cor}
$A = \{a_1, a_2, \dots, a_{n}\}$ is an extremal multi Skolem set if and only if there is a $\pi \in S_n$ with Toeplitz characteristic $l_t(\pi) = (a_1, a_2, \dots, a_{n})$.
\end{cor}
The proof of Proposition~\ref{prop:bij} also give an obvious bijection between extremal (multi) Skolem sequences and permutations for each pair of corresponding extremal (multi) Skolem sets and Toeplitz characteristics.  
\begin{cor}
The number of extremal (multi) Skolem sequences that can be generated from $A = \{a_1, a_2, \dots, a_{n}\}$ equals
$|\{\pi \in S_n \mid l_t(\pi) = (a_1, a_2, \dots, a_{n})\}|$. 
\end{cor}

Recall that a fixed-point-free involution $\pi \in S_{2n}$ is a permutation such that $[P_{\pi}]$ is symmetric and have $0$ trace (i.e., empty main diagonal). The following bijection between (not necessarily extremal) Skolem sets and binary Toeplitz X-rays of fixed-point-free involutions also appear implicitly in~\cite{Nordh08}. 
\begin{prop}\label{prop:nonext}
There is a bijection between Skolem sets of size $n$ and binary diagonal (Toeplitz) X-rays of fixed-point-free involutions in $S_{2n}$. More explicitly, $A = \{a_1,\dots,a_n\}$ is a Skolem set if and only if there is a fixed-point-free involution $\pi \in S_{2n}$ with Toeplitz characteristic $l_t(\pi) = (2n-a_n,2n-a_{n-1},\dots,2n-a_1,2n+a_1,2n+a_2,\dots,2n+a_n)$.
\end{prop}
\begin{proof}
$A = \{a_1, a_2, \dots, a_{n}\}$ is a
Skolem set if and only if the set $\{1,2,\dots, 2n\}$ can be partitioned into the differences in $A$, i.e., $a_i = t_i - s_i$, ($i = 1,2,\dots,n$), if and only if there is a fixed-point-free involution $\pi \in S_{2n}$ defined as $\pi(s_i) = t_i$ and $\pi(t_i) = s_i$, with Toeplitz characteristic $l_t(\pi) = (2n-a_n,2n-a_{n-1},\dots,2n-a_1,2n+a_1,2n+a_2,\dots,2n+a_n)$.
\end{proof}

\begin{ex}
Consider the fixed-point-free involution $\pi \in S_6$ given by $\pi(1) = 5, \pi(2) = 3, \pi(3) = 2, \pi(4) = 6, \pi(5) = 1, \pi(6) = 4$. The Toeplitz characterstic of $\pi$ is $l_t(\pi) = (2,4,5,7,8,10) = (6-4,6-2,6-1,6+1,6+2,6+4)$. The fixed-point-free involution $\pi$ defines a partition of $\{1,2,3,4,5,6\}$ into pairs $(t_i,s_i)$ by $\pi(s_i) = t_i$ and $\pi(t_i) = s_i$, with corresponding set of differences $A = \{t_i - s_i\}$ ($i = 1,2,3$) (in this example $s_1=2, s_2 = 4, s_3 = 1$, and hence, the differences are $A=\{\pi(2) - 2,\pi(4) - 4, \pi(1) - 1\} = \{1,2,4\}$ with corresponding Skolem sequence (4,1,1,2,4,2)).
$$\left[ \begin{array}{cccccc} 0 & 0 & 0 & 0 & t_3 & 0 \\ 0 & 0 & t_1 & 0 & 0 & 0\\ 0 & s_1 & 0 & 0 & 0 & 0\\ 0 & 0 & 0 & 0 & 0 & t_2 \\ s_3 & 0 & 0 & 0 & 0 & 0\\0 & 0 & 0 & s_2 & 0 & 0\\ \hline 4 & 1 & 1 & 2 & 4 & 2\end{array} \right]$$
\end{ex}

By Proposition~\ref{prop:nonext}, the following conjecture about binary Toeplitz X-rays of fixed-point-free involutions is equivalent to the main conjecture regarding Skolem sets (i.e., Conjecture~\ref{perfconj}).
\begin{conj}
A sequence of $2n$ positive integers $(a_1, a_2, \dots, a_{2n})$ with $a_1 < a_2 < \dots < a_{2n}$ and $a_{2n+1-i}-a_i = 2a_i$ ($i = 1,\dots,n$) is the Toeplitz characteristic $l_t(\pi)$ of a fixed-point-free involution $\pi \in S_{2n}$ 
if and only if $\sum^k_{i=1} a_i \geq k^2$ for each $1 \leq k \leq n$, 
and the number of even elements in $\{a_1,a_2,\dots,a_n\}$ is even.
\end{conj}

Note that the proof of Proposition~\ref{prop:nonext} hold without modifications for multi Skolem sets and non-binary Toeplitz X-rays.
\begin{cor}
$A = \{a_1, a_2, \dots, a_{n}\}$ is a multi Skolem set if and only if there is a fixed-point-free involution $\pi \in S_{2n}$ with Toeplitz characteristic $l_t(\pi) = (2n-a_n,2n-a_{n-1},\dots,2n-a_1,2n+a_1,2n+a_2,\dots,2n+a_n)$.
\end{cor}

As in the extremal case, the proof of Proposition~\ref{prop:nonext} give a bijection between (multi) Skolem sequences and fixed-point-free involutions for each pair of corresponding (multi) Skolem sets and Toeplitz characteristics.  
\begin{cor}
The number of (multi) Skolem sequences that can be generated from $A = \{a_1, a_2, \dots, a_{n}\}$ equals the number of fixed-point-free involutions $\pi \in S_{2n}$ with Toeplitz characteristic $l_t(\pi) = (2n-a_n,2n-a_{n-1},\dots,2n-a_1,2n+a_1,2n+a_2,\dots,2n+a_n)$.
\end{cor}

We conclude this section with the following problem, which is a modest step towards a proof of Conjecture~\ref{conj:post}.
\begin{prob}
Give a polynomial-time algorithm recognizing binary vectors that are Toeplitz X-rays of permutations.
\end{prob}
Note that the corresponding problem for non-binary vectors is NP-complete. If Conjecture~\ref{conj:post} holds, then such a polynomial algorithm for binary vectors follows directly.

\section{X-rays of doubly stochastic matrices}
A doubly stochastic (abbreviated DS) matrix $D$ is a square matrix, all of which elements are nonnegative real numbers, and
the elements in each column and row sum to $1$. In particular, any permutation matrix is a DS matrix. Diagonal X-rays of DS matrices are defined analogously to the definitions of diagonal X-rays of permutation matrices. Observe that the existence question for DS matrices having diagonal X-ray $(d_1,d_2,\dots,d_{2n-1})$ is in P, since the problem can be expressed as a LP (feasibility) problem. This is in sharp contrast with the corresponding problem for permutation matrices which is
NP-complete~\cite{Nordh09}. 
We refrain from discussing non-binary diagonal X-rays of DS matrices and diagonal X-rays of symmetric DS matrices, and focus on the existence question for binary diagonal X-rays of DS matrices.

Conjecture~\ref{extconj} regarding extremal Skolem sets and the equivalent Conjecture~\ref{conj:post} regarding binary diagonal X-rays of permutation matrices, corresponds 
to the following conjecture about binary diagonal X-rays of DS matrices. 
\begin{conj}
There is a DS matrix $D$ with binary diagonal X-ray $d(D) = (d_1,d_2,\dots,d_{2n-1})$ and corresponding Toeplitz characterstic $l_t(D) = (a_1,a_2,\dots,a_n)$ if and only if $\sum^k_{i=1} a_i \geq k^2$ for each $1 \leq k \leq n$ with equality holding for $k=n$.
\label{conj:stoc}
\end{conj}
To see that the density conditions in Conjecture~\ref{conj:stoc} are necessary, recall the following theorem due to Birkhoff~\cite{Birkhoff46}.
\begin{thm}[\cite{Birkhoff46}]
A nonnegative real matrix $D$ is doubly stochastic if and only if there exists permutation
matrices $P_1,\dots,P_m$ and positive real numbers $c_1,\dots,c_m$ summing to $1$ such that
$D = c_1P_1 + \cdots + c_mP_m$. 
\end{thm}
The necessity of the density conditions in Conjecture~\ref{conj:stoc} hence follows\footnote{By Birkhoff's theorem,
$D = c_1P_1 + \cdots + c_mP_m$ where $P_1,\dots,P_m$ are permutation matrices and $c_1 + \cdots + c_m = 1$.
By the necessary density conditions for diagonal X-rays of permutations (see~\cite{Postnikov,BF14,Nordh08}), 
$\sum^k_{i=1} a^j_i \geq k^2$ for each $1 \leq k \leq n$ with equality holding for $k=n$, where $l_t(P_j) = (a^j_1,\dots, a^j_n)$ for each 
permutation $P_j$ ($1 \leq j \leq m)$. 
Observe that $\sum^k_{i=1} a_i = \sum_{j=1}^m (c_j \sum^k_{i=1} a^j_i)$ for all $1 \leq k \leq n$, where
$l_t(D) = (a_1,\dots,a_n)$ is the Toeplitz characteristic of $D$.
Note that $\sum_{j=1}^m (c_j \sum^k_{i=1} a^j_i) \geq k^2$ for all $1 \leq k \leq n$,
since $\sum^k_{i=1} a^j_i \geq k^2$ for all $1 \leq k \leq n$ and the
$c_j$'s add to $1$. The fact that $\sum^n_{i=1} a_i = n^2$ follows from the fact that $\sum^n_{i=1} a^j_i = n^2$ for each of the $m$ permutation matrices.} from the necessity of the density conditions for the Toeplitz characteristic $l_t(\pi) = (a_1,a_2,\dots,a_n)$ of $\pi \in S_n$.

We conjecture that if there is a DS matrix with binary diagonal X-ray $(d_1,d_2,\dots,d_{2n-1})$ then there is a permutation matrix with the same diagonal X-ray.
\begin{conj}
There is a DS matrix $D$ of order $n$ with binary diagonal X-ray $d(D) = (d_1,d_2,\dots,d_{2n-1})$ if and only if there is a permutation $\pi \in S_n$ with binary diagonal X-ray $d(\pi) = (d_1,d_2,\dots,d_{2n-1})$.
\label{conj:stoctoperm}
\end{conj}
Note that the two conjectures above (Conjectures~\ref{conj:stoc} and \ref{conj:stoctoperm}) together 
imply Conjecture~\ref{conj:post} (and the equivalent Conjecture~\ref{extconj}).
\bibliographystyle{plain}
\bibliography{references}

\end{document}